\documentclass[11pt,reqno]{amsart}
\usepackage[T1]{fontenc}
\usepackage{lmodern}
\usepackage[expansion=false]{microtype}
\usepackage{mathtools}
\usepackage{amssymb}
\usepackage{booktabs}
\usepackage{url}
\usepackage[hidelinks]{hyperref}
\numberwithin{equation}{section}
\newtheorem{theorem}{Theorem}[section]
\newtheorem{proposition}[theorem]{Proposition}
\newtheorem{lemma}[theorem]{Lemma}
\theoremstyle{definition}
\newtheorem{definition}[theorem]{Definition}
\newcommand{\doi}[1]{\url{https://doi.org/#1}}
\newcommand{\lean}[1]{\texttt{#1}}
\title{The Borowiecki--Broere generalized total coloring
conjecture for planar graphs}
\author{James Alexander Schreib}\thanks{ORCID iD: \url{https://orcid.org/0009-0005-1048-0209}}
\address{New York University}
\email{jas10320@nyu.edu}
\date{September 5, 2026}
\subjclass[2020]{05C15, 05C10, 05B35}
\keywords{planar graph, total coloring, forest partition,
matroid partition, Four Color Theorem}
\begin{document}
\begin{abstract}
We prove the conjecture of Borowiecki and Broere that every finite
simple planar graph admits a generalized total coloring with four
colors: a proper vertex coloring with four colors together with an
edge coloring with four colors in which each edge color class is a
forest and no edge receives the color of either endpoint. In fact
every proper four-coloring of such a graph extends to an edge
coloring of the required kind. The proof combines a component-count
inequality for planar graphs whose vertices are partitioned into four
independent sets, obtained by completing to a triangulation, with the
matroid partition theorem applied to four graphic matroids.
A Lean~4 formalization relative to six named background assumptions
is described below.
\end{abstract}
\maketitle
\markboth{J. A. Schreib}{Borowiecki--Broere conjecture}
\section{Introduction}
In this paper, a \emph{generalized total coloring} with four colors is a
vertex coloring $\phi \colon V(G) \to \{1,2,3,4\}$ together with an
edge coloring $\psi \colon E(G) \to \{1,2,3,4\}$ such that $\phi$ is
proper, the subgraph with edge set $\psi^{-1}(i)$ is a forest, and
$\psi(e) \notin \{\phi(u),\phi(v)\}$ for every edge $e = uv$.
Borowiecki and Broere conjectured that every finite simple planar
graph admits such a coloring~\cite[Conjecture~3, p.~255]{BB16}.
This is the $(\mathcal O,\mathcal D_1)$ case of the generalized total
colorings introduced in~\cite{BKMM11}, where $\mathcal O$ is the class
of edgeless graphs and $\mathcal D_1$ is the class of forests.
The weaker conjecture allowing vertex color classes to induce forests
is also stated in~\cite[Conjecture~4]{BB16}; recent work of Cara and
Dorfling~\cite{CD26} establishes that version for two infinite families
of planar graphs.

The main result here is the following extension form, which implies
the conjecture directly.
\begin{theorem}\label{thm:main}
Let $G$ be a finite simple planar graph. Every proper vertex
coloring $\phi \colon V(G) \to \{1,2,3,4\}$ extends to an edge
coloring $\psi \colon E(G) \to \{1,2,3,4\}$ such that each
$\psi^{-1}(i)$ is a forest and no edge receives the color of either
endpoint.
\end{theorem}
The vertex coloring is supplied by the Four Color
Theorem~\cite{AH77,AHK77,RSST97}; Gonthier describes its formal proof
in~\cite{Gon08}. Given $\phi$, the edge coloring is produced
by the matroid partition theorem of Edmonds~\cite{Edm65} and
Edmonds--Fulkerson~\cite[Theorem~1c]{EF65} applied to four graphic matroids;
see also Oxley~\cite{Oxl11}. The new content is the rank inequality that makes the
criterion applicable. It is proved by completing the graph to a
triangulation while controlling monochromatic edges, then deleting
the added edges one at a time. Background on planar graphs, blocks,
and Euler\textquoteright s formula follows Diestel~\cite{Die17}.
Section~2 fixes notation. Section~3 proves the component estimate.
Section~4 derives Theorem~\ref{thm:main} by matroid partition.
Section~5 describes the conditional Lean~4 formalization, and
Section~6 records the computational resources used.
\section{Preliminaries}
All graphs are finite and, unless stated otherwise, simple. A planar
graph is considered with a fixed embedding in the sphere; its faces
are the connected components of the complement. A \emph{block} is a
maximal $2$-connected subgraph, a bridge with its ends, or an
isolated vertex. Every connected graph is the $1$-vertex join of its
blocks along cut vertices, and Euler\textquoteright s formula
$n - m + f = 2$ holds for each connected plane graph~\cite{Die17}.
A triangulation of the sphere with
$n \ge 3$ satisfies $3f = 2m$ and hence $m = 3n - 6$; the same count
holds for a connected triangular plane multigraph. Here triangular
means that every face, including the outer face, is bounded by a
$3$-cycle.
For a graph $J$ write $\kappa(J)$ for its number of connected
components, isolated vertices included, with $\kappa(\emptyset) = 0$.
A \emph{forest} is an acyclic graph; a pair of parallel edges already
forms a $2$-cycle and so never lies in a forest. Deleting a color
class $V_i$ from $H$ means deleting those vertices and all incident
edges.
\begin{definition}\label{def:gtc}
A \emph{generalized total coloring} of $G$ with four colors consists
of a proper $\phi \colon V(G) \to \{1,2,3,4\}$ and
$\psi \colon E(G) \to \{1,2,3,4\}$ such that each $\psi^{-1}(i)$
is the edge set of a forest and $\psi(uv) \notin \{\phi(u),\phi(v)\}$ for all
$uv \in E(G)$.
\end{definition}
We use graphic matroids throughout. If $M$ is the graphic matroid of
a graph with vertex set $V$, the rank of $A \subseteq E$ is
$|V| - \kappa(V,A)$, loops contributing nothing. Declaring prescribed
elements to be loops means taking the restriction to the remaining
elements and adjoining the prescribed elements as loops. For matroids
$M_1,\dots,M_4$ on a common ground set $E$, the partition
theorem states that $E$ splits as
$F_1 \mathbin{\dot\cup} \cdots \mathbin{\dot\cup} F_4$ with $F_i$
independent in $M_i$ if and only if
$|A| \le \sum_i r_i(A)$ for every $A \subseteq E$
\cite[Theorem~1c, p.~150]{EF65}.
\section{The component estimate}
\begin{lemma}\label{lem:alternating}
Let $D$ be a closed disk in the sphere, and let $a_1,b_1,a_2,b_2$
be distinct points in this cyclic order on $\partial D$.
Two simple arcs in $D$, one joining $a_1$ to $a_2$ and the other
joining $b_1$ to $b_2$, each meeting $\partial D$ only at its ends,
cannot be disjoint.
\end{lemma}
\begin{proof}
Let $P$ be the first arc. Join its ends by a simple arc $R$ whose
interior lies in the complementary open disk. Then $P\cup R$ is a
Jordan curve. At $a_1$ and $a_2$ it passes from one side of
$\partial D$ to the other. Thus the two open boundary arcs of
$\partial D\setminus\{a_1,a_2\}$ lie on opposite sides of $P\cup R$.
In particular, $b_1$ and $b_2$ lie on opposite sides. By the Jordan
curve theorem, an arc joining them must meet $P\cup R$. An arc in
$D$ disjoint from $P$ cannot meet $R$, whose interior is outside $D$
and whose ends belong to $P$. This is a contradiction.
\end{proof}

\begin{lemma}\label{lem:del}
Let $T$ be a connected loopless triangular plane multigraph, and let
$S\subseteq V(T)$ induce a forest, with parallel edges counted as
distinct edges. Then $T-S$ is nonempty and connected.
\end{lemma}
\begin{proof}
The complement is nonempty: otherwise $T$ would be a forest, whereas
each of its facial $3$-cycles is a cycle in $T$.
Root each tree of $T[S]$ and delete its vertices with parents before
children. When deleting $v$, the only possible previously deleted
neighbor in $T[S]$ is its parent. There is at most one edge to that
parent, since two parallel edges would form a cycle in $T[S]$.

Consider the cyclic list of incident edge ends at $v$ in the
\emph{original} triangulation $T$, recording the corresponding
neighbors with multiplicity. Consecutive entries are joined by the
third edge of an original facial triangle. At most one entry
corresponds to a deleted vertex. Removing that entry, if present,
leaves a linear sequence, or a cyclic list if none was removed,
joining all surviving neighbors of $v$. Every joining edge remains
present because both its ends survive, and none is incident with $v$.
Hence all surviving neighbors are connected without $v$.
Inductively the graph before deleting $v$ is connected, so each
component after the deletion contains a surviving neighbor of $v$.
The remaining graph is nonempty because it contains $T-S$.
It is therefore connected, completing the induction.
\end{proof}

\begin{proposition}\label{prop:rank}
Let $H$ be a finite simple planar graph with a proper vertex coloring
with classes $V_1,\dots,V_4$. Then
\begin{equation}\label{eq:comp}
\sum_{i=1}^4 \kappa(H - V_i) \le 3\lvert V(H)\rvert - \lvert E(H)\rvert.
\end{equation}
\end{proposition}
\begin{proof}
Let $n = \lvert V(H)\rvert$ and $m = \lvert E(H)\rvert$.
The small cases are direct: the empty graph gives $0 \le 0$; an
isolated vertex gives $3 \le 3$; a single edge gives $4 \le 5$.
Hence assume $H$ is $2$-connected with $n \ge 3$, the general case
being recovered by block induction at the end. Fix an embedding in
the sphere; every face boundary is then a simple cycle. The auxiliary
construction below permits parallel edges, always treated as distinct
edges.

Inside each face, repeatedly insert a diagonal joining differently
colored nonconsecutive boundary vertices while possible. Each
insertion splits one polygonal face into two and preserves properness.
An inserted diagonal may be parallel to an edge elsewhere in the
embedding. The process terminates because each insertion decreases
the sum of the face lengths minus three, summed over all faces, by one.
When the process stops, every face is a triangle or an alternating
bichromatic quadrilateral: on a facial cycle $v_1\cdots v_k$ with
$k \ge 5$, both $v_1v_3$ and $v_1v_4$ are diagonals, and they cannot
both be monochromatic lest $v_3$ and $v_4$ share a color; a
quadrilateral with no bichromatic diagonal has equal opposite pairs.

Choose one diagonal in each remaining quadrilateral so that the chosen
diagonals form a forest $F$. Process the quadrilaterals in any order.
The two diagonals of $a_1b_1a_2b_2$ are monochromatic, in distinct
colors. If both choices closed a cycle, $F$ would already contain an
$a_1$--$a_2$ simple path and a $b_1$--$b_2$ simple path. These paths
are vertex disjoint, since components of $F$ are monochromatic.
All previously chosen diagonals were drawn in other faces, so both
paths lie in the closed disk complementary to the open quadrilateral
face. The interior vertices of either path cannot be the other two
corners, which have a different color, or its own ends, since the path
is simple. No edge of $F$ is a boundary edge of this quadrilateral:
its boundary edges are bichromatic. Since edges in the embedding meet
only at common ends, each path meets the boundary only at its ends.
Lemma~\ref{lem:alternating} now gives a contradiction.
So some diagonal keeps $F$ a forest.

Let $T$ be the resulting triangular plane multigraph. Every
monochromatic edge of $T$ lies in $F$, so each $T[V_i]$ is a forest.
Lemma~\ref{lem:del} gives $\kappa(T - V_i) = 1$ for $i = 1,\dots,4$.

Delete the added diagonals to recover $H$. In each original face, the
triangles of its triangulation have a tree as internal dual; root it
and delete the diagonals in order of increasing depth. When $xy$ is
deleted, its child-side triangle $xyz$ still has both other sides
present, as boundary edges or deeper diagonals. For $i$ distinct from
the colors of $x$ and $y$, the path $xzy$ is present unless $i$ is the
color of $z$, so deleting $xy$ does not add a component; for a color
of an endpoint, $xy$ is already absent. Only the graph obtained by
deleting the color of $z$ can gain a component, and deleting one edge
increases the number of components by at most one. Each deletion
therefore raises $\sum_i \kappa$ by at most one.

With $d$ added diagonals,
$\sum_i \kappa(H - V_i) \le 4 + d$. Euler\textquoteright s formula with
$3f = 2\lvert E(T)\rvert$ gives $\lvert E(T)\rvert = 3n - 6$, also in
the multigraph case, so $d = 3n - 6 - m$ and the sum is at most
$3n - m - 2$, whence~\eqref{eq:comp} for the block.

Joining two graphs at one vertex lowers the right-hand side of
\eqref{eq:comp} by $3$ against the sum of the two sides, while the
left-hand side drops by exactly $3$: deleting the shared
vertex\textquoteright s color leaves the graphs disjoint, and for each
other color the retained copies of that vertex identify one component
from each side. Disjoint unions add on both sides. Starting from the
small cases, block decomposition and disjoint unions extend
\eqref{eq:comp} to every $H$.
\end{proof}
\section{Matroid partition and the coloring}
\begin{proof}[Proof of Theorem~\ref{thm:main}]
Let $G$ be finite simple planar and fix any proper four-coloring
$\phi$, with $V_i = \phi^{-1}(i)$. For each $i$ let $M_i$
be the graphic matroid on the edges avoiding $V_i$, all other edges
declared loops; its independent sets are exactly the forests meeting
no edge incident with $V_i$. For $A \subseteq E(G)$ put
$H = (V(G),A)$, so $\phi$ stays proper on $H$ and
$r_i(A) = \lvert V(G) \setminus V_i\rvert - \kappa(H - V_i)$.
Summing and applying Proposition~\ref{prop:rank} to $H$,
\begin{equation}\label{eq:part}
\sum_{i=1}^4 r_i(A)
= 3\lvert V(G)\rvert - \sum_{i=1}^4 \kappa(H - V_i)
\ge \lvert A\rvert \qquad \text{for all $A \subseteq E(G)$.}
\end{equation}
By the partition criterion recalled in Section~2,
\eqref{eq:part} yields
$E(G) = F_1 \mathbin{\dot\cup} \cdots \mathbin{\dot\cup} F_4$ with
$F_i$ independent in $M_i$. Keep $\phi$ on vertices and put color $i$
on $F_i$. Each edge class is a forest avoiding vertices of its color,
and $\phi$ remains proper. This is
Definition~\ref{def:gtc} for $G$.
\end{proof}
\section{Formal verification}\label{sec:formal}
The theorem \path{BorowieckiBroere.generalized_total_coloring}
is formalized in Lean~4 relative to exactly six custom axioms:
\begin{enumerate}
\item \path{four_color_theorem};
\item \path{matroid_partition};
\item \path{sphereMap_of_twoConnected_planeDrawing};
\item \path{jordan_alternating_paths};
\item \path{sphere_triangulation_induced_forest_complement};
\item \path{sphere_triangulation_triangle_deletion_order}.
\end{enumerate}
The first two are the Four Color Theorem and the graphic specialization of
matroid partition. The remaining four isolate the passage from a plane
drawing to a sphere map, the Jordan separation used in
Lemma~\ref{lem:alternating}, the forest-complement connectivity of
Lemma~\ref{lem:del}, and the dual deletion order used in the proof of
Proposition~\ref{prop:rank}.

Conditional on those results, Lean proves every other step: colored face
completion, the monochromatic diagonal forest, triangular component counts,
deletion bookkeeping, reduction across components and cut vertices, the
planar component inequality and rank condition, and the final edge-forest
partition with endpoint-color avoidance. Thus, relative to the six axioms
above, the formalization is sufficient to derive the theorem stated here.
The extension statement of Theorem~\ref{thm:main} is formalized as
\path{BorowieckiBroere.extends_of_proper_four_coloring}; it uses the
same background axioms except for \path{four_color_theorem}, which
is needed only to obtain the existence consequence.

The guarded \lean{\#print axioms} checks in \lean{AxiomAudit.lean} confirm
that the public endpoint uses exactly these six custom axioms. No
\lean{sorry}, \lean{admit}, or \lean{native\_decide} occurs. The pinned
Lean~4.30.0 and mathlib~v4.30.0 \lean{lake build} reproduces the check.
The artifact is available at
\url{https://github.com/jamesschreib/borowiecki-broere-conjecture},
at commit
\href{https://github.com/jamesschreib/borowiecki-broere-conjecture/tree/31975d021eb780b578834143ca890315de38e977}{\texttt{31975d021eb7}}.
\section{Tool and computational resource disclosure}
ChatGPT~6 Astra (OpenAI) was accessed in chat to generate the proof.
ChatGPT~5.6 Sol was accessed through Codex to write portions of the
Lean~4 development described in Section~\ref{sec:formal}.
For recommendations on disclosing such tool use, see the Leiden
Declaration on Artificial Intelligence and Mathematics~\cite{Lei26}.
All statements and proofs were subsequently
verified independently by the author. The formal development is
checked by the Lean kernel, and the Lean statements have been
confirmed by hand to faithfully restate the corresponding statements
of this paper. The author alone is responsible for the correctness of
the arguments and for the accuracy and completeness of the citations.

\end{document}